\documentclass[12pt,cd]{amsart}

\usepackage{amsfonts, amssymb,amsmath, amsthm, amsxtra, latexsym, mathrsfs, amscd}
\usepackage{times, float}
\usepackage{graphics}
\usepackage{hyperref}
\usepackage[latin1]{inputenc}
\usepackage[arrow, matrix, curve]{xy}

\theoremstyle{plain}
\newtheorem{theo+}{Theorem}[section]
\newtheorem{prop+}[theo+]{Proposition}
\newtheorem{coro+}[theo+]{Corollary}
\newtheorem{lemm+} [theo+]{Lemma}
\newtheorem{deep+}  [theo+]  {Deep Result}
\newtheorem{fact+}  [theo+]  {Fact}
\theoremstyle{definition}
\newtheorem{exam+}  [theo+]  {Example}
\newtheorem{rema+}  [theo+]  {Remark}
\newtheorem{defi+}  [theo+]  {Definition}
\newtheorem{xca+}[theo+]{Exercise}
\newenvironment{theorem}{\begin{theo+}}{\end{theo+}}

\newenvironment{remark}{\begin{rema+}}{\end{rema+}}

\numberwithin{equation}{section}
\usepackage{color}

\def\fa{\mathfrak a}

\def\fg{\mathfrak g}

\def\fn{\mathfrak n}

\def\fv{\mathfrak v}
\def\fz{\mathfrak z}

\def\norm#1{\Vert#1\Vert}
\def\Norm#1_#2{\Vert#1\Vert_{#2}}

\def\ad{\operatorname{ad}}

\def\Ind{\operatorname{Ind}}

\def\exp{\operatorname{exp}}

\def\SO{\operatorname{SO}}

\def\tr{\operatorname{tr}}

\def\End{\operatorname{End}}

\def\im{\operatorname{Im}}
\def\id{\text{id}}
\def\blank{\,\cdot\,}
\def\1{{\bf 1}}
\def\0{{\bf 0}}

\def\bc{\mathbb C}
\def\br{\mathbb R}
\def\bn{\mathbb N}

\def\bh{\mathbb H}
\def\bo{\mathbb O}

\def\boldf{\mathbb F}

\begin{document}

\title[Invariant differential operators on H-type groups]{Invariant differential operators on H-type groups and discrete components in restrictions of complementary series of rank one semisimple groups}

\author{J. M\"{o}llers}
\address{Department of Mathematics, Aarhus University, Ny Munkegade 118, DK-8000 Aarhus C, Denmark}
\email{\it moellers@imf.au.dk}

\author{B. {\O}rsted}
\address{Department of Mathematics, Aarhus University, Ny Munkegade 118, DK-8000 Aarhus C, Denmark}
\email{\it orsted@imf.au.dk}

\author{G. Zhang}
\address{Mathematical Sciences, Chalmers University of Technology and Mathematical Sciences, G\"oteborg University, SE-412 96 G\"oteborg, Sweden}
\email{\it genkai@chalmers.se}
\thanks{Research by G. Zhang partially supported by the Swedish Science Council (VR)}

\begin{abstract}
We explicitly construct a finite number of discrete components in the restriction of complementary series representations of rank one semisimple groups $G$ to rank one subgroups $G_1$. For this we use the realizations of complementary series representations of $G$ and $G_1$ on Sobolev type spaces on the nilpotent radicals $N$ and $N_1$ of the minimal parabolics in $G$ and $G_1$, respectively. The groups $N$ and $N_1$ are of H-type and we construct explicitly invariant differential operators between $N$ and $N_1$. These operators induce the projections onto the discrete components.\newline
Our construction of the invariant differential operators is carried out uniformly in the framework of H-type groups and also works for those H-type groups which do not occur as nilpotent radical of a parabolic subgroup in a semisimple group.

\medskip\noindent \emph{2010 MSC:} Primary 22E46; Secondary 22E25.

\medskip\noindent \emph{Key words and phrases:} Lie groups,
  complementary series, Sobolev type spaces, invariant
  differential operators.
\end{abstract}

\maketitle

\baselineskip 1.25pc

\section{Introduction}

The study of branching rules is one of the central
topics in representation theory. Given a unitary representation
$\pi$ of a Lie group $G$  and a closed subgroup $G_1\subseteq G$
the objective is to find the decomposition of $\pi$ into
irreducible representations of $G_1$. This decomposition consists in general
of a discrete and a continuous part; for various analytic considerations
it is of importance to find the discrete components in this decomposition.
The probably most studied class of representations in this context is the
discrete series, see e.g. the work of Duflo--Vargas~\cite{DV10},
Gross--Wallach~\cite{GW00}, Kobayashi~\cite{Kob05,Kob08}, {\O}rsted--Vargas~\cite{OV04}, Zhang~\cite{Zha02} and references
therein.

There is another somewhat opposite
class of representations, the complementary
series, which is of substancial interest in 
the spectral theory of locally symmetric spaces, in particular rank-one spaces.
Recently various authors studied discrete components in the restriction
of complementary series of rank one groups to symmetric
subgroups, see the work of Kobayashi--Speh~\cite{KS}, M\"{o}llers--Oshima~\cite{MO}, Speh--Venkataramana~\cite{SV11}, Speh--Zhang~\cite{SZ} and Zhang~\cite{Zha11}. For rank one orthogonal groups
the discrete spectrum is known by \cite{MO} and can explicitly be constructed in
terms of Juhl's covariant differential operators~\cite{Juh09}, see \cite{KS,MO,SZ}.
For the other rank one groups Speh--Zhang~\cite{SZ} found finitely
many discrete components in the restriction,
without providing an explicit construction of these components.

In the present paper we shall find explicit embeddings for all
discrete components given in \cite{SZ}. For this we use the
realization of complementary series representations on certain
function spaces on a nilpotent group $N$, the nilpotent radical of a
minimal parabolic. For rank-one groups the nilpotent groups in
question are H-type groups, a generalization of the Heisenberg
group. We explicitly construct invariant differential restriction
operators between $N$ and the H-type subgroup $N_1\subseteq N$
corresponding to $G_1$. These operators realize the projections onto
the discrete components. Our construction is carried out uniformly in
the framework of H-type groups and even provides invariant
differential restriction operators between complementary series
representations of harmonic extensions of H-type groups studied
earlier in \cite{Cow82, CK84, DZ99}. We remark that in particular the
Heisenberg group $N=\bc^n\ltimes\br$ and its subgroups
$N_1=\bc^m\ltimes\br$ are special cases of our setting.

Let us explain our results in more detail.

\subsection{Complementary series of H-type groups} Let $\fn=\fv +\fz$
be the Lie algebra of an H-type group $N$ with $\fz$ being
the center acting on $\fv$ as a Clifford
module. There are essentially two $N$-invariant
second order differential operators on
$N$, the Kohn sub-Laplacian $\Delta$
defined by vector fields in $\fv$, and the
central Laplacian $\Box$ defined by vector fields in $\fz$.
The complementary series representations $(\pi_\nu,X_\nu)$, the parameter $\nu$ being in a certain interval,
are unitary representations of the harmonic extension $S=N\rtimes A$ of $N$ where $A=\mathbb{R}_+$ acts on $N$ by dilations. They can be realized on certain Sobolev type spaces $X_\nu$ on $N$ with norm given by $u\mapsto\|\Delta^{\frac{\rho-\nu}{2}}u\|_{L^2(N)}$ where $\rho=\frac{1}{2}(\dim\fv+2\dim\fz)$ (note that this is not the $S$-invariant norm but equivalent to it, see Section~\ref{sec:ComplSerOfS} for details).

For an H-type subgroup $N_1\subseteq N$ (see Section~\ref{sec:HtypeSubgroups} for the precise definition) denote by $\Delta_1$ the Kohn sub-Laplacian on its Lie algebra $\fn_1$ and by $(\pi_\nu^\flat,X_\nu^\flat)$ the corresponding complementary series representations of the harmonic extension $S_1=N_1\rtimes A\subseteq S$.

\subsection{Invariant differential operators} In the present paper we
construct a meromorphic family $\mathcal{D}_{\nu,k}$, $\nu\in\bc$,
$k\in\bn_0$, of differential restriction operators $C_c^\infty(N)\to
C_c^\infty(N_1)$ which are polynomial in $\Delta$, $\Delta_1$ and
$\Box$ and intertwine the representations $\pi_\nu$ and
$\pi_{\nu+2k}^\flat$ (see Proposition~\ref{prop:DiffOpOnKsigma} and
Theorem~\ref{thm:IntertwiningDiffOps}). We further show that for
certain parameters $\nu$ and $k$ the operators $\mathcal{D}_{\nu,k}$
are bounded with respect to the respective Sobolev type norms and hence
induce intertwining operators $X_\nu\to X_{\nu+2k}^\flat$ between the
Hilbert spaces (see Theorem~\ref{thm:Boundedness}).

\subsection{Application to branching rules} In the case where $NA$ is
the Iwasawa subgroup of a rank one semisimple Lie group $G$,
i.e. $G=NAK$ with $K\subseteq G$ maximal compact, the complementary
series representations $X_\nu$ of $NA$ extend to irreducible unitary
representations of $G$. Assume that $N_1A$ is the Iwasawa subgroup of
a subgroup $G_1$ of $G$. Then the adjoint operators
$\mathcal{D}_{\nu,k}^*:X_{\nu+2k}^\flat\to X_\nu$ are actually
$G_1$-equivariant isometric embeddings and hence realize $X_{\nu+2k}^\flat$ explicitly as subrepresentations of $X_\nu$ (see Theorem~\ref{thm:BranchingRankOne}). We remark that one can also view this embedding in the dual picture: Since $\mathcal{D}_{\nu,k}$ is a differential restriction operator its dual operator $\mathcal{D}_{\nu,k}':(X_{\nu+2k}^\flat)'\to X_\nu'$ embeds $(X_{\nu+2k}^\flat)'$ as distributions on $N$ supported on the subspace $N_1$.

\subsection{Relation to previous results} In the degenerate case where
$\fn=\br^n$ is abelian and $\fn=\br^{n-1}$ the operators
$\mathcal{D}_{\nu,k}$ were constructed before by Juhl~\cite{Juh09} and
generalized to arbitrary signature in \cite{KOSS} (see
Remark~\ref{rem:JuhlOps} for details). They were used to construct
discrete components in the restriction of complementary series of
$SO(1,n+1)$ to $SO(1,n)$ by Kobayashi--Speh~\cite{KS} and
M\"{o}llers--Oshima~\cite{MO}. 
For other rank one groups the abstract existence of the discrete components in the branching rule was previously established by Speh--Zhang~\cite{SZ} without giving an explicit embedding.
This way of obtaining intertwining operators 
has also been investigated recently; see e.g. \cite{BC12, BSKZ14, KS, Zha}. 
Our  
invariant differential operator is roughly speaking a
meromorphic continuation of a family of operators obtained  
by a composition of some $G_1$-intertwining multiplication operators   
and the Knapp-Stein integral intertwining operator.

\section{Harmonic analysis on H-type groups}

We recall some facts about H-type groups, harmonic extensions and their representation theory. For details we refer the reader to \cite[Section 3]{CH89} and references therein.

\subsection{H-type groups}

Let $\fn$ be a Lie algebra over $\br$ endowed with an inner product
$\langle\cdot,\cdot\rangle$ with corresponding norm $|\cdot|$. Assume
that there exists an orthogonal decomposition $\fn=\fv+\fz$ with
$[\fv,\fv]=\fz$ and $[\fv,\fz]=0$ so that $\fn$ is either abelian or
two-step nilpotent. We denote
$$
p=\dim\fv, \quad q=\dim\fz, \quad
\rho = \frac{p+2q}{2}.
$$
The integer $2\rho$ is also called the homogeneous dimension of
$N$. For every $z\in\fz$ the identity
$$ \langle J_ZX,Y\rangle = \langle Z,[X,Y]\rangle, \qquad X,Y\in\fv, $$
defines a skew-symmetric linear map $J_Z:\fv\to\fv$. The nilpotent Lie algebra $\fn$ and its corresponding nilpotent Lie group $N$ are called of H-type if
$$ J_Z^2 = -16|Z|^2. $$
(Note that in the literature normally the convention $J_Z^2=-|Z|^2$ is used which amounts to a simple scaling of the center. However, for our formulas and calculations the above convention is more convenient.)

We note that this condition can be reformulated in terms of Clifford
modules. In fact the space $\fv$ is a module over the Clifford algebra
$C(\fz)$ associated to $(\fz,\langle\cdot,\cdot\rangle)$. For each $\fz$ there is up to
isomorphism only one irreducible Clifford module $\fv_0$ and hence
$\fv\simeq\fv_0^n$. In Table~\ref{tab:IrrCliffordModules} we list the
dimensions of $\fv_0$. 
However, we will not use this classification in the rest of this article; we only note that for a given $\fv$ one can find all of its
Clifford submodules and thus its H-type subalgebras (see Section~\ref{sec:HtypeSubgroups}) that
we study in the present paper.

\begin{table}[h]
\begin{tabular}
{
|l|l|l|l|l|l|l|l|l|
}
\hline
 $q=\dim \fz$  &
 $8k$ & $8k+1$&
 $8k+2$
&$8k+3$ &$8k+4$ &$8k+5$&  $8k+6$&$8k+7$ 
\\
\hline
 $d=\dim \fv_0$  & $2^{4k}$
 &$2^{4k+1}$  & $2^{4k+2}$ & $2^{4k+2}$  & $2^{4k+3}$  &$2^{4k+3}$  &$2^{4k+3}$ 
&$2^{4k+3}$ 
\\ 
\hline
\end{tabular}
\vskip0.1cm
 \caption{Dimensions of irreducible Clifford modules}\label{tab:IrrCliffordModules}
\end{table}

\begin{exam+}
\begin{enumerate}
\item For $\mathbb{F}=\br,\bc,\bh$ denote by $X\mapsto\overline{X}$ the canonical involution on $\mathbb{F}$. For $X=(X_1,\ldots,X_n)\in\mathbb{F}^n$ we write $X^*=(\overline{X_1},\ldots,\overline{X_n})^t$. We endow $\fv=\mathbb{F}^n$ with the inner product $\langle X,Y\rangle=\Re(XY^*)$. Further, put $\fz=\im\mathbb{F}=\{X\in\mathbb{F}:X+\overline{X}=0\}$, endowed with the restriction of the inner product on $\mathbb{F}$. Then the space $\fn=\fv+\fz=\mathbb{F}^n\oplus\im\mathbb{F}$ turns into an H-type Lie algebra if endowed with the Lie bracket
$$ [(X,Z),(Y,W)] = (0,2(XY^*-YX^*)). $$
The corresponding map $J:\fz\to\End(\fv)$ is given by
$$ J_ZX = -4Z\cdot X. $$
The same construction also works for $\mathbb{F}=\bo$ and $n=1$ where $\fn=\bo\oplus\im\bo$. These H-type groups all appear as nilpotent part in the Iwasawa decomposition of rank one semisimple Lie groups (see Section~\ref{sec:RankOneGroups}). Note that the irreducible Clifford module of $C(\fz)$ is $\fv_0=\mathbb{F}$ in these cases.
\item An H-type group which is not the nilpotent part in the Iwasawa
  decomposition of a semisimple Lie group is given by
  $\fn=\fv\oplus\fz=\bh^{m+n}\oplus\im\bh$,
$m, n \ge 1$, with corresponding map $J:\fz\to\End(\fv)$ given by (cf. \cite[Section 1]{CDKR91})
$$ J_ZX = -4(Z\cdot X_1,\ldots,Z\cdot X_m,X_{m+1}\cdot Z,\ldots,X_{m+n}\cdot Z). $$
\end{enumerate}
\end{exam+}

We identify the group $N$ with its Lie algebra $\fn$ using the group product
$$ (X,Z)\cdot (Y,W) := (X+Y,Z+W+\tfrac{1}{2}[X,Y]), \qquad X,Y\in\fv,Z,W\in\fz. $$
In what follows we use capital letters $(X,Z)$ for an element of the Lie algebra and small letters $n$ for an element of the Lie group. The Haar measure $dn$ on $N$ will be normalized by
$$ \int_N f(n) dn = k(p,q)^{-1} \int_\fv \int_\fz f(X,Z) dZ dX, $$
where $dX$ and $dZ$ are the Lebesgue measures on $\fv$ and $\fz$ normalized by the inner product and
$$ k(p,q) = \frac{\pi^{\frac{p+q+1}{2}}2^{1-p-q}}{\Gamma(\frac{p+q+1}{2})}. $$

We need some basic identities for H-type groups:
\begin{align*}
 |J_ZX|^2 &= \langle J_ZX,J_ZX\rangle = -\langle J_Z^2X,X\rangle = 16|Z|^2|X|^2,\\
 J_{[X,X']}X &= 16|X|^2P_{J_\fz X}X',
\end{align*}
where $P_{J_\fz X}$ denotes the orthogonal projection of $\fv$ onto the subspace $J_\fz X$.

\subsection{Solvable extensions of H-type groups}


Let $A=\mathbb R_+$ act on $N$ by $t\cdot(X,Z)=(tX,t^2Z)$.
Then the semi-direct product $S=N\rtimes A$ is a Lie group with
the group product
$$ (X,Z,t)\cdot(Y,W,s) = (X+tY,Z+t^2W+\frac{1}{2}t[X,Y], t\cdot s) $$
We will write $S=NA$ for short and call $S$ the solvable  extension of
$N$.
$S$ is sometimes also called the harmonic extension as
it is a homogeneous harmonic Riemannian manifold
\cite{Damek-Ricci-bams}.




\subsection{H-type subgroups}\label{sec:HtypeSubgroups}

Let $\fv=\fv_1+\fv_2$ be an orthogonal decomposition such that $[\fv_1,\fv_2]=0$ and $[\fv_1,\fv_1]=\fz$. (Such a $\fv_1$ is sometimes also called a Clifford submodule, see~\cite{CDKR91}.) Then $\fn_1=\fv_1+\fz$ is an H-type subalgebra of $\fn$. Denote by $N_1$ the corresponding subgroup of $N$ and by $S_1=N_1A$ the harmonic extension of $N_1$ which is a closed subgroup of $S=NA$.
We put
$$ p_1=\dim \fv_1, \quad p_2=\dim \fv_2,  \quad \rho_1 = \frac{p_1+2q}{2}.  $$


\subsection{Differential operators}

For $X\in \fn$ the corresponding left-invariant differential operator
will also be denoted by $X$, namely
$$ (Xf)(n)=\frac{d}{dt}(f(n\exp(tX))|_{t=0}. $$
On the Lie algebra level the left-invariant differential operators are given by
\begin{align*}
 Sf(X,Z) &= \partial_Sf(X,Z)-\frac{1}{2}\partial_{[S,X]}f(X,Z), && S\in\fv,\\
 Tf(X,Z) &= \partial_Tf(X,Z), && T\in\fz.
\end{align*}

Let $(S_j)_{j=1,\ldots,p_1}\subseteq\fv_1$ be an orthonormal basis of $\fv_1$ and $(S_j)_{j=p_1+1,\ldots,p}\subseteq\fv_2$ an orthonormal basis of $\fv_2$. The Kohn sub-Laplacians for $\fn_1=\fv_1+\fz$ and $\fn_2=\fv_2+\fz$ are given by
$$ \Delta_1=S_1^2+\cdots+S_{p_1}^2, \qquad \Delta_2=S_{p_1+1}^2+\cdots+S_p^2. $$
The Kohn sub-Laplacian for $\fn$ is the sum
$$ \Delta=\Delta_1 +\Delta_2. $$
We also need the central Laplacian which is given by
$$ \Box=T_1^2 +\cdots +T_q^2, $$
where $(T_j)_j\subseteq\fz$ is an orthonormal basis of $\fz$.

\subsection{Plancherel formula}\label{sec:PlancherelFormula}

We recall very briefly the Plancherel formula on $N$. For $\fz=\{0\}$ the group $N$ is abelian and the Plancherel formula is the classical Plancherel formula for the Euclidean Fourier transform on $\fn=\fv$. Therefore we assume $\fz\neq\{0\}$.

For each $\mu\in\fz\setminus\{0\}$ the map $\frac1{4|\mu|} J_\mu$ is a complex structure on $\fv$. Let $\fv_\mu$ denote the space $\fv$ endowed with this complex structure and $\mathcal{O}(\fv_\mu)$ the space of holomorphic functions on $\fv_\mu$. Further, let $\{\blank,\blank\}_\mu$ denote the corresponding Hermitian form
$$ \{X.Y\}_\mu := \langle X,Y\rangle-i\langle\tfrac{1}{4|\mu|}J_\mu X,Y\rangle. $$
The Fock space
$$ \mathcal F_\mu(\fv) = \left\{\xi\in\mathcal{O}(\fv_\mu):\|\xi\|_\mu^2:=\int_{\fv}|\xi(X)|^2e^{-2|\mu|\cdot|X|^2}dX<\infty\right\} $$
carries an irreducible unitary representation $\sigma_\mu$ of $N$ and $\fn$ given by
$$ (\sigma_\mu(X,Z)\xi)(Y) = \exp(-|\mu|[|X|^2+2\{Y,X\}_\mu]))\exp(i\langle\mu,Z\rangle)\xi(Y+X). $$
Identifying $\fz^\ast$ with $\fz$ via the inner product the representation $\sigma_\mu$ has infinitesimal character $i\mu$. We remark that the space $\mathcal{P}(\fv_\mu)$ of holomorphic polynomials on $\fv_\mu$ is dense in $\mathcal{F}_\mu(\fv)$ and the subspaces $\mathcal{P}_m(\fv_\mu)$, $m\geq0$, of homogeneous polynomials of degree $m$ are mutually orthogonal.

The representation $\sigma_\mu$ of $N$ induces a map,
the Weyl transform, $\sigma_\mu:L^1(N)\to\End(\mathcal{F}_\mu(\fv))$ by
$$ \sigma_\mu(f) = \int_N f(n) \sigma_\mu(n^{-1}) dn. $$
This definition is different from the usual convention; the integrand is $\sigma_\mu(n^{-1})$ instead of $\sigma_\mu(n)$. Our definition has the advantage that for left-invariant vector fields $X\in\fn$ we have the following formula preserving the order:
$$ \sigma_\mu(Xf)=\sigma_\mu(X)\sigma_\mu(f). $$
If we define the convolution of $f, g\in L^1(N)$ by
$$ (f\ast g)(n') = \int_N f(n^{-1}n') g(n) dn = \int_N f(n)g(n'n^{-1}) dn $$
then we further have
$$ \sigma_\mu(f\ast g) = \sigma_\mu(f)\circ\sigma_\mu(g). $$

The following inversion and Plancherel formula for the Weyl transform hold, at least for $f\in C_c^\infty(N)$:
$$ f(n) = c(p,q)\int_\fz \tr(\sigma_\mu(f)\sigma_\mu(n)) |\mu|^{\frac{p}{2}} d\mu =  c(p,q)\int_\fz \tr(\sigma_\mu(n)\sigma_\mu(f)) |\mu|^{\frac{p}{2}} d\mu $$
and
\begin{equation}
 \int_{N}|f(n)|^2dn = c(p,q)\int_{\fz}\Vert\sigma_\mu(f)\Vert_{HS}^2|\mu|^{\frac{p}{2}} d\mu,\label{eq:PlancherelFormula}
\end{equation}
where
$$ \norm{T}_{HS}^2 = \langle T,T\rangle_{HS} = \tr(TT^*), $$
the Hilbert-Schmidt norm of an operator $T$ on $\mathcal F_{\mu}(\fv)$, and
$$ c(p,q) = \frac{2^{1-\frac{p}{2}-2q}\pi^{\frac{1-q}{2}}}{\Gamma(\frac{p+q+1}{2})}. $$

It is easy to see that
\begin{equation}
 \sigma_\mu(\Delta)|_{\mathcal{P}_m(\fv_\mu)} = -2|\mu|(4m+p)\label{eq:WeylTrafoDelta}
\end{equation}
and
\begin{equation}
 \sigma_\mu(\Box) = -|\mu|^2.\label{eq:WeylTrafoBox}
\end{equation}

Note that the space $\mathcal{P}(\fv_\mu)$ is the tensor product
$\mathcal{P}(\fv_{1,\mu})\otimes\mathcal{P}(\fv_{2,\mu})$ where
$\fv_{i,\mu}$ is the space $\fv_i$ endowed with the complex structure
$\frac{1}{4|\mu|}J_\mu|_{\fv_i}$, $i=1,2$. We denote by $\mathcal
F_\mu(\fv_1)$ the corresponding Fock space on $\fv_1$ and by
$\sigma_\mu^{(1)}$ 
the Weyl transform for the subgroup $N_1$.
From \eqref{eq:WeylTrafoDelta} it also follows that
\begin{align}
\begin{split}
 \sigma_\mu(\Delta_1)|_{\mathcal{P}_m(\fv_{1,\mu})\otimes\mathcal{P}(\fv_{2,\mu})} &= -2|\mu|(4m+p_1),\\
 \sigma_\mu(\Delta_2)|_{\mathcal{P}(\fv_{1,\mu})\otimes\mathcal{P}_m(\fv_{2,\mu})} &= -2|\mu|(4m+p_2),
\end{split}\label{eq:WeylTrafoDelta12}
\end{align}
where $\mathcal{P}_m(\fv_{i,\mu})$, $i=1,2$, denotes the subspace of $\mathcal{P}(\fv_{i,\mu})$ of homogeneous polynomials of degree $m$.

\subsection{Weyl transform of powers of the norm function}

Let
$$ K(X,Z)=|X|^4+ |Z|^2, \qquad  X\in\fv,Z\in\fz,$$
denote the norm function on $\fn$. For $s\in\bc$ we define
$$ K_s(n) := K(n)^s, \qquad n\in N. $$
The Weyl transform $\sigma_\mu(K_s)$ of the function $K_s$ acts as a
scalar on the space of homogeneous polynomials of a fixed degree
$m$. More precisely we have (see \cite{Cow82})
\begin{equation}
 \sigma_\mu(K_s)|_{\mathcal{P}_m(\fv_\mu)} = \frac{2^\rho\Gamma(\frac{p+q+1}{2})\Gamma(\frac{2m-2s-\rho+\frac{p}{2}+1}{2})\Gamma(2s+\rho)}{\Gamma(-s)\Gamma(\frac{-2s-\rho+\frac{p}{2}+1}{2})\Gamma(\frac{2m+2s+\rho+\frac{p}{2}+1}{2})}|\mu|^{-2s-\rho}.\label{eq:WeylTrafoNormFct}
\end{equation}

\subsection{Unitary representations of the group $S=NA$}\label{sec:ComplSerOfS}

For $\nu\in\mathbb C$ denote by $\chi_\nu$ the character of $A=\mathbb{R}_+$ given by $\chi_\nu(t)=t^{-\nu}$. Consider the induced representation $Ind_A^{NA}(\chi_\nu)$ realized by the left-regular action on the space
$$ \{f\in C^\infty(NA):f(ga)=\chi_\nu(a)^{-1} f(g)\,\forall\,g\in NA,a\in A\}. $$
Restricting functions to $N$ gives an isomorphic realization $\pi_\nu$ on the space $C^\infty(N)$ which can be described by
$$ \pi_\nu(n)f(n')=f(n^{-1}n'), \quad \pi_\nu(t)f(X,Z)=t^{-\nu} f(t^{-1}X,t^{-2}Z). $$
It restricts to a smooth representation $(X_\nu^\infty,\pi_\nu)$ of $NA$ on $X_\nu^\infty=C_c^\infty(N)$ by the same formulas. For $\nu\in\rho +i\mathbb R$ the representation can be extended to $X_\nu=L^2(N)$ and is unitary on this space.

In \cite{DZ99} Dooley and Zhang determine all real values $\nu\in\mathbb R$ for which the representation $\pi_\nu$ is unitarizable. For this let
$$ \Lambda_{\nu}f := K_{\frac{\nu}{2}-\rho}\ast f, $$
as meromorphic family of operators on $C_c^\infty(N)$. Denote by $\langle\cdot,\cdot\rangle$ the standard inner product on $L^2(N)$.

\begin{theo+}
[{\cite{Cow82}}, {\cite[Theorem 1.1]{DZ99}}]
\label{thm:ComplSerForHType}
The form on $C_c^\infty(N)$ given by
\begin{equation}
 (f, f)_{\nu}=\langle \Lambda_{\nu} f, f \rangle\label{eq:InvForm}
\end{equation}
is $S$-invariant under $\pi_\nu$, i.e.
$$ (\pi_\nu(g)f,\pi_\nu(g)f)_\nu = (f,f)_\nu \qquad \forall\,g\in S, f\in C_c^\infty(N). $$
It is positive definite if and only if
$$ \nu\in \begin{cases}(0,2\rho) & \mbox{for $q=0$,}\\(\rho-\tfrac{p}{2}-1,\rho+\tfrac{p}{2}+1) & \mbox{for $q>0$.}\end{cases} $$
In this case the completion $X_\nu$ of $C_c^\infty(N)$ with 
respect to the norm $\Vert\!\cdot\!\Vert_\nu^2=(\cdot, \cdot)_\nu$ forms a unitary representation of $S=NA$ called complementary series.
\end{theo+}

\begin{remark}\label{rem:CSNormOnFTSide}
Using the Plancherel formula~\eqref{eq:PlancherelFormula} together with \eqref{eq:WeylTrafoNormFct} we find
\begin{align*}
 \|f\|_\nu^2 &= \int_N (K_{\frac{\nu}{2}-\rho}*f)(n)\overline{f(n)} dn\\
 &= c(p,q) \int_\fz \langle\sigma_\mu(K_{\frac{\nu}{2}-\rho}*f),\sigma_\mu(f)\rangle_{HS}|\mu|^{\frac{p}{2}} d\mu\\
 &= c(p,q) \int_\fz \langle\sigma_\mu(K_{\frac{\nu}{2}-\rho})\circ\sigma_\mu(f),\sigma_\mu(f)\rangle_{HS}|\mu|^{\frac{p}{2}} d\mu\\
 &= c(p,q) \sum_m \lambda_m(\nu) \int_\fz \|P_m\circ\sigma_\mu(f)\|_{HS}^2 |\mu|^{p+q-\nu} d\mu,
\end{align*}
where
$$ \lambda_m(\nu) = \frac{2^r\Gamma(\frac{p+q+1}{2})\Gamma(\frac{2m-\nu+\rho+\frac{p}{2}+1}{2})\Gamma(\nu)}{\Gamma(\rho-\frac{\nu}{2})\Gamma(\frac{-\nu+\rho+\frac{p}{2}+1}{2})\Gamma(\frac{2m+\nu-\rho+\frac{p}{2}+1}{2})} \sim (1+m)^{\rho-\nu}. $$
Observe that we may generally define certain Sobolev type norms on the nilpotent group $N$ by replacing $\lambda_m(\nu)$ by the standard weight sequence $(1+m)^s$ and study corresponding analytic problems such as boundedness of restriction operators and Sobolev type embedding theorems. Note that by \eqref{eq:WeylTrafoDelta} the norm $u\mapsto\|u\|_\nu$ is equivalent to the norm $u\mapsto\|\Delta^{\frac{\rho-\nu}{2}}u\|_{L^2(N)}$.
\end{remark}

Similarly we denote by $X_\nu^{\flat,\infty}$ the corresponding family of smooth representations of the group $S_1=N_1A$ on $C_c^\infty(N_1)$ and by $X_\nu^\flat$ their unitary completions for $\nu$ in the interval of the complementary series.


\section{Invariant differential operators}

We construct a sequence of left-invariant differential operators $D_{s,k}$, $k\geq0$, on $N$ which induce $S_1$-intertwining operators $\mathcal{D}_{\nu,k}:X_\nu^\infty\to X_{\nu+2k}^{\flat,\infty}$.

\subsection{Identities for the norm function}

First we apply various operators to the functions $K_s$. Some of these
computations are already done in \cite[Theorem 2.8]{CK84}
and we present here the full details. Let $\fn_1$
be as in Section~\ref{sec:HtypeSubgroups} an H-type subalgebra.

For $X\in\fv$ we use the notation $X=X_1+X_2$ with $X_1\in\fv_1$ and $X_2\in\fv_2$.

\begin{lemm+}\label{lem:OperatorsOnKsigma}
\begin{enumerate}
\item $\Delta K_s=B(s)|X|^2K_{s-1}$ with $B(s)=4s(4s+2\rho-2)$,
\item $\Delta^2K_s=B(s)(2pK_{s-1}+(B(s-1)+16(s-1))|X|^4K_{s-2})$,
\item $\Box K_s=2s qK_{s-1}+4s(s-1)|Z|^2K_{s-2}$
\item $\Delta_2 K_s=(4s p_2|X|^2+8s(2s+q-1)|X_2|^2)K_{s-1}$.
\end{enumerate}
\end{lemm+}

\begin{proof}
We have
\begin{align*}
 S_jK_s ={}& s(4\langle X,S_j\rangle|X|^2-\langle J_ZS_j,X\rangle)K_{s-1},\\
 S_j^2K_s ={}& s(s-1)(4\langle X,S_j\rangle|X|^2-\langle J_ZS_j,X\rangle)^2K_{s-2}\\
 & + s(4|X|^2+8\langle X,S_j\rangle^2+\tfrac{1}{2}\langle J_{[X,S_j]}X,S_j\rangle)K_{s-1}.
\end{align*}
Summing over $1\leq j\leq p$ we obtain
\begin{align*}
 \Delta K_s ={}& s(s-1)(16|X|^6-8|X|^2\langle J_ZX,X\rangle+|J_ZX|^2)K_{s-2}\\
 & +s\left(4p|X|^2+8|X|^2+\frac{1}{2}\sum_{j=1}^p\langle J_{[X,S_j]}X,S_j\rangle\right)K_{s-1}
\end{align*}
Since $\langle J_ZX,X\rangle=0$, $|J_ZX|^2=16|Z|^2|X|^2$ and
$$ \sum_{j=1}^p\langle J_{[X,S_j]}X,S_j\rangle=16|X|^2\sum_{j=1}^p\langle P_{J_\fz X}S_j,S_j\rangle=16|X|^2\tr(P_{J_\fz X})=16q|X|^2 $$
this gives
\begin{align*}
 \Delta K_s ={}& 16s(s-1)|X|^2K_{s-1}+4s\left(2\rho+2\right)|X|^2K_{s-1}\\
 ={}& B(s)|X|^2K_{s-1}.
\end{align*}
Summing over $p_1+1\leq j\leq p$ instead gives
\begin{align*}
 \Delta_2 K_s ={}& s(s-1)\left(16|X_2|^2|X|^4-8|X|^2\langle J_ZX_2,X\rangle+\sum_{j=p_1+1}^p\langle J_ZX_2,S_j\rangle^2\right)K_{s-2}\\
 & +s\left(4p_2|X|^2+8|X_2|^2+\frac{1}{2}\sum_{j=p_1+1}^p\langle J_{[X,S_j]}X,S_j\rangle\right)K_{s-1}
\intertext{Note that $J_ZX_2\in\fv_2$ and hence $\langle J_ZX_2,X\rangle=\langle J_ZX_2,X_2\rangle=0$. Further, $\langle J_ZX_2,S_j\rangle=0$ for $j=1,\ldots,p_1$ and therefore $\sum_{j=p_1+1}^p\langle J_ZX_2,S_j\rangle^2=|J_ZX_2|^2=16|Z|^2|X_2|^2$. Finally, $\sum_{j=p_1+1}^p\langle J_{[X,S_j]}X,S_j\rangle=\sum_{j=p_1+1}^p\langle J_{[X_2,S_j]}X_2,S_j\rangle=16q|X_2|^2$ as above. Together this gives}
 ={}& 16s(s-1)|X_2|^2K_{s-1}+4s p_2|X|^2K_{s-1}+8s(q+1)|X_2|^2K_{s-1}\\
 ={}& (4s p_2|X|^2+8s(2s+q-1)|X_2|^2)K_{s-1}.
\end{align*}
To calculate $\Delta^2K_s$ we first note that $\Delta|X|^2=2p$. Applying $\Delta$ to $|X|^2K_{s-1}$ gives
\begin{align*}
 \Delta^2K_s ={}& B(s)\Bigg(B(s-1)|X|^4K_{s-2}+2pK_{s-1}\\
 & \hspace{1.5cm}+2\sum_{j=1}^p2\langle X,S_j\rangle(s-1)(4\langle X,S_j\rangle|X|^2-\langle J_ZS_j,X\rangle)K_{s-2}\Bigg)\\
 ={}& B(s)(2pK_{s-1}+(B(s-1)+16(s-1))K_{s-2}).
\end{align*}
Finally
\begin{align*}
 T_jK_s &= 2s\langle Z,T_j\rangle K_{s-1},\\
 T_j^2K_s &= 2s K_{s-1}+4s(s-1)\langle Z,T_j\rangle^2K_{s-2}
\intertext{and hence}
 \Box K_s &= 2qs K_{s-1}+4s(s-1)|Z|^2K_{s-2}.\qedhere
\end{align*}
\end{proof}

\begin{coro+}\label{cor:OpIdentitiesForKsigma}
\begin{enumerate}
\item $(\Delta^2 +4(4s +2\rho -2)^2\Box)K_s=2B(s)(2s+q-1)(4s+2\rho-4)K_{s-1}$,
\item $(\Delta_2 -\frac{p_2}{4s +2\rho-2}\Delta)K_s=8s(2s+q-1)|X_2|^2K_{s-1}$
\end{enumerate}
\end{coro+}

Next we study the action of the Kohn sub-Laplacians on functions of the form $|X_2|^{2m}K_{s-m}$.

\begin{lemm+}
\begin{enumerate}
\item $\Delta(|X_2|^{2m}K_{s-m})=4(s-m)(4s+2\rho-2)|X|^2|X_2|^{2m}K_{s-m-1}+2m(2m+p_2-2)|X_2|^{2m-2}K_{s-m}$,
\item $\Delta_2(|X_2|^{2m}K_{s-m})=4(s-m)(4m+p_2)|X|^2|X_2|^{2m}K_{s-m-1}+8(s-m)(2s+q-2m-1)|X_2|^{2m+2}K_{s-m-1}+2m(2m+p_2-2)|X_2|^{2m-2}K_{s-m}$.
\end{enumerate}
\end{lemm+}

\begin{proof}
To calculate $\Delta(|X_2|^{2m}K_{s-m})$ we first note that $\Delta|X_2|^{2m}=\Delta_2|X_2|^{2m}=2m(2m+p_2-2)|X_2|^{2m-2}$. Hence we obtain with Lemma~\ref{lem:OperatorsOnKsigma}~(1)
\begin{align*}
 \Delta(|X_2|^{2m}K_{s-m}) ={}& B(s-m)|X|^2|X_2|^{2m}K_{s-m-1}+2m(2m+p_2-2)|X_2|^{2m-2}K_{s-m}\\
 & +2\sum_{j=p_1+1}^p 2m\langle X,S_j\rangle|X_2|^{2m-2}\cdot(s-m)(4\langle X,S_j\rangle|X|^2-\langle J_ZS_j,X\rangle)K_{s-m-1}\\
 ={}& B(s-m)|X|^2|X_2|^{2m}K_{s-m-1}+2m(2m+p_2-2)|X_2|^{2m-2}K_{s-m}\\
 & +4m(s-m)(4|X_2|^2|X|^2-\langle J_ZX_2,X\rangle)|X_2|^{2m-2}K_{s-m-1}.\\
\end{align*}
Since $\langle J_ZX_2,X\rangle=\langle J_ZX_2,X_2\rangle=0$ this shows (1). For (2) we calculate with Lemma~\ref{lem:OperatorsOnKsigma}~(4) and the previous calculation:
\begin{align*}
 \Delta_2(|X_2|^{2m}K_{s-m}) ={}& (4(s-m)p_2|X|^2+8(s-m)(2s+q-2m-1)|X_2|^2)|X_2|^{2m}K_{s-m-1}\\
 & +2m(2m+p_2-2)|X_2|^{2m-2}K_{s-m}+16m(s-m)|X|^2|X_2|^{2m}K_{s-m-1}
\end{align*}
which shows (2).
\end{proof}

\begin{coro+}\label{cor:OpIdentitiesforKsigmaX2m}
\begin{multline*}
 ((4s+2\rho-2)\Delta_2-(4m+p_2)\Delta)(|X_2|^{2m}K_{s-m})\\
 = 8(s-m)(4s+2\rho-2)(2s+q-2m-1)|X_2|^{2m+2}K_{s-m-1}\\
 +2m(2m+p_2-2)(4s+2\rho_1-4m-2)|X_2|^{2m-2}K_{s-m}.
 \end{multline*}
\end{coro+}

\subsection{Construction of the differential operators}

We now construct constant coefficient differential operators $D_{s,k}$ inductively by putting
\begin{align*}
 D_{s,0} &:= 1,\\
 D_{s,1} &:= \frac{1}{8s(4s+2\rho-2)(2s+q-1)} \Big((4s+2\rho-2)\Delta_2-p_2\Delta\Big),
\end{align*}
and
\begin{multline}
 D_{s,k+1} := \frac{1}{8(s-k)(4s+2\rho-2)(2s+q-2k-1)} \Bigg[ \Big((4s+2\rho-2)\Delta_2-(4k+p_2)\Delta\Big)D_{s,k}\\
 -\frac{2k(2k+p_2-2)(4s+2\rho_1-4k-2)}{2B(s)(2s+q-1)(4s+2\rho-4)}\Big(\Delta^2+4(4s+2\rho-2)^2\Box\Big)D_{s-1,k-1} \Bigg].\label{eq:DefinitionInvariantDiffOps}
\end{multline}
We note that the operator $D_{s,k}$ is defined in particular if $-2s\in(\rho-\frac{p}{2}-1,\rho)$ and $-2s+2k\in(\rho_1-\frac{p_1}{2}-1,\rho_1)$.

\begin{prop+}\label{prop:DiffOpOnKsigma}
The constant coefficient differential operator $D_{s,k}$ is of the form
$$ D_{s,k} = \sum_{i+j+2\ell=k}C_{s,k}(i,j,\ell)\Delta_1^i\Delta_2^j\Box^\ell. $$
Further, $D_{s,k}$ satisfies
\begin{equation*}
 D_{s,k}K_s = |X_2|^{2k}K_{s-k}.
\end{equation*}
\end{prop+}

\begin{proof}
Only the last statement is non-trivial. We prove it by induction on $k$. For $k=0$ this is trivial and for $k=1$ this is Corollary~\ref{cor:OpIdentitiesForKsigma}~(2). Now suppose we have $D_{t,\ell}K_t=|X_2|^{2\ell}K_{t-\ell}$ for every $t$ and $0\leq \ell\leq k$. We apply this identity for $(t,\ell)=(s,k)$ and $(t,\ell)=(s-1,k-1)$ to the identity in Corollary~\ref{cor:OpIdentitiesforKsigmaX2m} to obtain
\begin{multline*}
 ((4s+2\rho-2)\Delta_2-(4k+p_2)\Delta)D_{s,k}K_s\\
 = 8(s-k)(4s+2\rho-2)(2s+q-2k-1)|X_2|^{2k+2}K_{s-k-1}\\
 +2k(2k+p_2-2)(4s+2\rho_1-4k-2)D_{s-1,k-1}K_{s-1}.
\end{multline*}
By Corollary~\ref{cor:OpIdentitiesForKsigma}~(1) we can write $K_{s-1}$ as a differential operator consisting of $\Delta$ and $\Box$ applied to $K_s$. Inserting this into the above identity and isolating $|X_2|^{2k+2}K_{s-k-1}$ gives the claim.
\end{proof}

\subsection{Intertwining property}

In this section we show that Proposition~\ref{prop:DiffOpOnKsigma} implies an intertwining property for the operators $D_{s,k}$. For this let $R:C_c^\infty(\fn)\to C_c^\infty(\fn_1)$ denote the restriction operator.

\begin{theo+}\label{thm:IntertwiningDiffOps}
For each $k\geq0$ the operator $\mathcal{D}_{\nu,k}:=R\circ D_{-\frac{\nu}{2},k}$ is $S_1$-intertwining between $(X_{\nu}^{\infty}, \pi_\nu)|_{S_1}\to (X_{\nu+2k}^{\flat,\infty}, \pi_{\nu+2k}^\flat)$.
\end{theo+}

For this consider the restriction operator $R$ and the multiplication operator
$$ M_k:C_c^\infty(\fn)\to C_c^\infty(\fn), \quad M_kf(X,Z) = |X_2|^{2k}f(X,Z). $$

\begin{lemm+}
$R$ is an $S_1$-intertwining operator $X_{\beta}^\infty\to X_{\beta}^{\flat,\infty}$ and $M_k$ an $S_1$-intertwining operator $X_{\beta}^\infty\to X_{\beta-2k}^{\flat,\infty}$.
\end{lemm+}

Recall the meromorphic family of Knapp--Stein intertwining operators $\Lambda_\nu:X_{\nu}^\infty\to X_{2\rho-\nu}^\infty$ which are $S$-equivariant and are isomorphisms for $\nu$ in the interval of the complementary series (see Theorem~\ref{thm:ComplSerForHType}). The composition $\Lambda_{\nu}\circ\Lambda_{2\rho-\nu}:X_{2\rho-\nu}^\infty\to X_{2\rho-\nu}^\infty$ is a scalar multiple of the identity:
$$ \Lambda_{\nu}\circ\Lambda_{2\rho-\nu} = a(\nu)\cdot\id_{X_{2\rho-\nu}} $$
with some meromorphic function $a(\nu)$ which is non-zero for $\nu$ in the interval of the complementary series.

The following proposition together with the previous lemma implies Theorem~\ref{thm:IntertwiningDiffOps}:

\begin{prop+}\label{prop:DiffOpIntertwining}
We have, as meromorphic identity in $\nu$:
\begin{equation}
 \mathcal{D}_{\nu,k}= a(\nu)^{-1}\cdot R\circ\Lambda_{2\rho-\nu-2k}\circ M_k\circ\Lambda_\nu.\label{eq:DiffOpAsConjugationOfMult}
\end{equation}
\end{prop+}

\begin{rema+}
The idea of obtaining an operator which is intertwining for a subgroup $S_1$ by ``conjugating'' an intertwining multiplication operator $M_k$ by classical Knapp--Stein operators $\Lambda_\nu$ and $\Lambda_{2\rho-\nu-2k}$ was used before by Beckmann--Clerc~\cite{BC12}. In their case the underlying nilpotent group $N$ was abelian and they could use the Euclidean Fourier transform to turn the convolution operators $\Lambda_\nu$ and $\Lambda_{2\rho-\nu-2k}$ into multiplication operators, simplifying the calculations. This technique does not work in our case.
\end{rema+}

\begin{proof}[{Proof of Proposition~\ref{prop:DiffOpIntertwining}}]
Since both sides of the identity \eqref{eq:DiffOpAsConjugationOfMult} are meromorphic in $\nu$ it suffices to prove it for $\nu$ in the interval of the complementary series (see Theorem~\ref{thm:ComplSerForHType}). Here $\Lambda_{2\rho-\nu}$ is an isomorphism and we may multiply \eqref{eq:DiffOpAsConjugationOfMult} from the right with it and equivalently show
\begin{equation}
 \mathcal{D}_{\nu,k}\circ\Lambda_{2\rho-\nu} = R\circ\Lambda_{2\rho-\nu-2k}\circ M_k.\label{eq:RewrittenDiffOpAsConjugationOfMult}
\end{equation}
We first calculate the right hand side of \eqref{eq:RewrittenDiffOpAsConjugationOfMult}:
\begin{align*}
 (R\circ\Lambda_{2\rho-\nu-2k}\circ M_k)F(n_1) &= \Lambda_{2\rho-\nu-2k}M_kF(n_1)\\
 &= \int_N K_{-\frac{\nu}{2}-k}(n^{-1}n_1)|X_2|^{2k}F(n) dn,
\end{align*}
where $n=(X,Z)$. To calculate the left hand side we apply $D_{-\frac{\nu}{2},k}\circ\Lambda_{2\rho-\nu}$ to $F$ and find
\begin{align*}
 (D_{-\frac{\nu}{2},k}\circ\Lambda_{2\rho-\nu})F(n) &= \int_N (D_{-\frac{\nu}{2},k})_nK_{-\frac{1}{2}\nu}(n'^{-1}n)F(n') dn'.
\end{align*}
Since $D_{-\frac{\nu}{2},k}$ is left-invariant we find by Proposition~\ref{prop:DiffOpOnKsigma}
\begin{align*}
 (D_{-\frac{\nu}{2},k})_nK_{-\frac{1}{2}\nu}(n'^{-1}n) &= (D_{-\frac{\nu}{2},k}K_{-\frac{1}{2}\nu})(n'^{-1}n)\\
 &= |(n'^{-1}n)_2|^{2k}K_{-\frac{1}{2}\nu-k}(n'^{-1}n),
\end{align*}
where $(n'^{-1}n)_2$ is the $\fv_2$-component of $n'^{-1}n$. Restricting to $n=n_1\in N_1$ gives $|(n'^{-1}n)_2|=|X_2'|$ with $n'=(X',Z')$ and $X'=X_1'+X_2'$, $X_j'\in\fv_j$. Hence we obtain
\begin{equation}
 (R\circ D_{-\frac{\nu}{2},k}\circ\Lambda_{2\rho-\nu})F(n_1) = \int_N K_{-\frac{1}{2}\nu-k}(n'^{-1}n_1)|X_2'|^{2k}F(n') dn'.\label{eq:DiffOpsAndGenKernels}
\end{equation}
Hence the claimed identity follows.
\end{proof}

\begin{rema+}
Equation~\eqref{eq:DiffOpsAndGenKernels} shows that the operator $\mathcal{D}_{\nu,k}\circ\Lambda_{2\rho-\nu}$ is the integral operator $C_c^\infty(N)\to C^\infty(N_1)$ given by the kernel $K_{-\frac{1}{2}\nu-k}(n^{-1}n_1)|X_2|^{2k}$ with $n=(X_1+X_2,Z)\in N$, $n_1\in N_1$. These kernels were for $\fn=\br^n$ first constructed by Kobayashi--Speh~\cite{KS} and then generalized by M\"{o}llers--{\O}rsted--Oshima~\cite{MOO} to more general situations including the cases $\fn=\mathbb{F}^n\oplus\im\mathbb{F}$ with $\mathbb{F}=\br,\bc,\bh$, $n\geq1$ and $\mathbb{F}=\bo$, $n=1$.
\end{rema+}

\begin{rema+}\label{rem:JuhlOps}
In the degenerate case $\fn=\fv=\mathbb R^p$ and $\fn_1=\fv_1=\mathbb{R}^{p_1}$ we have $\Delta_1=\frac{\partial^2}{\partial x_1^2}+\cdots+\frac{\partial^2}{\partial x_{p_1}^2}$ and $\Delta_2=\frac{\partial^2}{\partial x_{p_1+1}^2}+\cdots+\frac{\partial^2}{\partial x_p^2}$. Further $\Box=0$ and hence the operators $D_{s,k}$ are basically one-variable polynomials. They can be described using the classical Jacobi polynomials $P_n^{(\alpha,\beta)}(z)$, see Appendix~\ref{app:JacobiPolys} for their definition and properties. More precisely, define two-variable polynomials $P_n^{(\alpha,\beta)}(z,w)$ by
$$ P_n^{(\alpha,\beta)}(z,w) := w^n P_n^{(\alpha,\beta)}\left(2\frac{z}{w}+1\right), $$
then
$$ D_{s,k} = \frac{k!}{2^{2k}(-2s)_{2k}(-2s-\frac{p}{2}+1)_k} P_k^{(\frac{p_2}{2}-1,-2s-\frac{p}{2})}(\Delta_2,\Delta_1). $$
This can be seen using the recurrence relation \eqref{eq:DefinitionInvariantDiffOps} and the following identity for the polynomials $P_n^{(\alpha,\beta)}(z,w)$ which follows from \eqref{eq:NewIdentityJacobiPolynomials}:
\begin{multline*}
 (n+\beta+1)\left[(2n+\alpha+\beta+2)z+(2n+\alpha+1)w\right]P_n^{(\alpha,\beta)}(z,w)\\
 = (n+\alpha)(2n+\alpha+\beta+2)(z+w)^2P_{n-1}^{(\alpha,\beta+2)}(z,w) + (n+1)(\beta+1)P_{n+1}^{(\alpha,\beta)}(z,w).
\end{multline*}
In the special case $p_1=p-1$ the Jacobi polynomials reduce to Gegenbauer polynomials (see \eqref{eq:JacobiSpecializationGegenbauer}) and the operators $D_{s,k}$ can be written as
$$ D_{s,k} = \frac{(2k)!}{2^{2k}(-2s)_{2k}(-4s-p+1)_{2k}} C_k^{-2s-\frac{p-1}{2}}\left(\Delta_2,\Delta_1\right), $$
where $C_n^\lambda(z,w)$ is the homogeneous polynomial of order $n$ which is defined using the classical Gegenbauer polynomials $C_n^\lambda(z)$ as follows:
$$ C_n^\lambda(z,w) := w^nC_{2n}^\lambda\left(i\sqrt{\frac{z}{w}}\right). $$
The operators $C_n^\lambda(\Delta_2,\Delta_1)$ were first constructed by Juhl~\cite{Juh09} (see also \cite{KOSS,KS}). We further remark that the polynomials $P_n^{(\alpha,\beta)}(z,w)$ and $C_n^\lambda(z,w)$ also appear in the description of intertwining differential operators for holomorphic symmetric pairs by Kobayashi--Pevzner~\cite{KP} (note the different definition of $C_n^\lambda(z,w)$).
\end{rema+}

\section{Boundedness}

This section is devoted to the proof of our main result:

\begin{theorem}\label{thm:Boundedness}
Let $\nu>0$, $k\in\bn_0$. If $\rho-\frac{p}{2}-1<\nu<\rho$ and $\rho_1-\frac{p_1}{2}-1<\nu+2k<\rho_1$ then the operator $\mathcal{D}_{\nu,k}:X_\nu^\infty\to X_{\nu+2k}^{\flat,\infty}$ is a bounded $S_1$-intertwining operator and hence extends to a bounded intertwining operator $X_\nu\to X_{\nu+2k}^\flat$ between Hilbert spaces.
\end{theorem}

For the abelian case $\fn=\fv=\br^n$ and $\fn_1=\fv_1=\br^m$ these results were already obtained by Kobayashi--Speh~\cite{KS} and by M\"{o}llers--Oshima~\cite{MO}. Hence we assume $\fz\neq\{0\}$ for the rest of this section which allows us to use the Plancherel formula for $N$ an $N_1$ as formulated in Section~\ref{sec:PlancherelFormula} as well as Remark~\ref{rem:CSNormOnFTSide}.

\subsection{Operator-valued partial trace} 

We first study the composition of the Weyl transform with the restiction operator $R:X_\nu^\infty\to X_\nu^{\flat,\infty}$. For this we introduce the operator-valued partial trace. For a Hilbert space $\mathcal{H}$ let $\mathcal{S}(\mathcal{H})$ denote the space of operators $T$ on $\mathcal{H}$ such that for some (or equivalently all) orthonormal bases $(e_\alpha)$ of $\mathcal{H}$ the sequence $\langle Te_\alpha,e_\alpha\rangle$ is rapidly decreasing. Further, for two Hilbert space $\mathcal{H}_i$, $i=1,2$, denote by $\mathcal{H}_1\otimes\mathcal{H}_2$ the Hilbert space tensor product.

\begin{lemm+}
The identity
$$ \langle\tr^{(2)}(T), S\rangle = \langle T, S\otimes I\rangle $$
defines a map
$$ \tr^{(2)}: \mathcal S(\mathcal H_1\otimes \mathcal H_2)\to\mathcal S(\mathcal H_1) $$
called the \textit{operator-valued partial trace}. It has the following properties:
\begin{enumerate}
\item $\tr^{(2)}((S\otimes I)\circ T)=S\circ\tr^{(2)}(T)$,
\item $\tr(\tr^{(2)}(T))=\tr(T)$.
\end{enumerate}
\end{lemm+}

We can express the Weyl transform of a restriction in terms of the operator-valued partial trace:

\begin{lemm+}\label{lem:WeylTrafoOfRestriction}
$$ \sigma_\mu^{(1)}(Rf) = \frac{\Gamma(\frac{p_1+q+1}{2})}{2^{\frac{p-p_1}{2}}\Gamma(\frac{p+q+1}{2})}|\mu|^{\frac{p-p_1}{2}}\cdot\tr^{(2)}(\sigma_\mu(f)). $$
\end{lemm+}

\begin{proof}
We use the Fourier inversion formula to find that for $n_1\in N_1$:
\begin{align*}
 Rf(n_1) &= f(n_1) = c(p,q) \int_\fz \tr(\sigma_\mu(n_1)\circ\sigma_\mu(f)) |\mu|^{\frac{p}{2}} d\mu\\
 &= c(p,q) \int_\fz \tr((\sigma_\mu^{(1)}(n_1)\otimes I)\circ\sigma_\mu(f)) |\mu|^{\frac{p}{2}} d\mu\\
 &= c(p,q) \int_\fz \tr(\sigma_\mu^{(1)}(n_1)\circ\tr^{(2)}(\sigma_\mu(f))) |\mu|^{\frac{p}{2}} d\mu.
\end{align*}
On the other hand
$$ Rf(n_1) = c(p_1,q) \int_\fz \tr(\sigma_\mu^{(1)}(n_1)\circ\sigma_\mu^{(1)}(Rf)) |\mu|^{\frac{p_1}{2}} d\mu $$
whence
\begin{equation*}
 \sigma_\mu^{(1)}(Rf) = \frac{c(p,q)}{c(p_1,q)}|\mu|^{\frac{p-p_1}{2}}\tr^{(2)}(\sigma_\mu(f)). \qedhere
\end{equation*}
\end{proof}

\subsection{Proof of boundedness}

Now we show the boundedness of the intertwining operators $\mathcal{D}_{\nu,k}$. We shall need an elementary fact (see also \cite[Lemma 3.5]{Zha11}):

\begin{lemm+}\label{lem:StandardEstimate}
Suppose that $\alpha>-1$, $\beta\geq0$ and $\beta-\alpha>1$. Then there exists a constant $C>0$ such that
$$ \sum_{n=0}^\infty\frac{(n+1)^\alpha}{(n+q+1)^\beta} \leq C\frac{1}{(q+1)^{\beta-\alpha-1}} \qquad \forall\,q\geq0. $$
\end{lemm+}


\begin{theo+}
For $\nu\in(\rho-\frac{p}{2}-1,\rho)$, $\nu+2k\in(\rho_1-\frac{p_1}{2}-1,\rho_1)$ the operator $\mathcal{D}_{\nu,k}:X_\nu^\infty\to X_{\nu+2k}^{\flat,\infty}$ is bounded.
\end{theo+}

\begin{proof}
By Remark~\ref{rem:CSNormOnFTSide} it suffices to show that
$$ \sum_n \lambda^\flat_n(\nu+2k)\|P_n^{(1)}\circ\sigma_\mu^{(1)}(\mathcal{D}_{\nu,k}f)\|_{HS}^2 \leq C|\mu|^{2k+p-p_1} \sum_m \lambda_m(\nu)\|P_m\circ\sigma_\mu(f)\|_{HS}^2. $$
By Lemma~\ref{lem:WeylTrafoOfRestriction} this is equivalent to showing
$$ \sum_n \lambda^\flat_n(\nu+2k)\|P_n^{(1)}\circ\tr^{(2)}(\sigma_\mu(D_{-\frac{\nu}{2},k}f))\|_{HS}^2 \leq C|\mu|^{2k}\sum_m \lambda_m(\nu)\|P_m\circ\sigma_\mu(f)\|_{HS}^2. $$
We calculate
\begin{align*}
 P_n^{(1)}\circ\tr^{(2)}(\sigma_\mu(D_{-\frac{\nu}{2},k}f)) &= \tr^{(2)}((P_n^{(1)}\otimes\1)\sigma_\mu(D_{-\frac{\nu}{2},k}f))\\
 &= \sum_{N\geq n} \tr^{(2)}((P_n^{(1)}\otimes P_{N-n}^{(2)})\sigma_\mu(D_{-\frac{\nu}{2},k})\sigma_\mu(f)).
\end{align*}
Using Proposition~\ref{prop:DiffOpOnKsigma} together with \eqref{eq:WeylTrafoBox} and \eqref{eq:WeylTrafoDelta12} we find that
$$ (P_n^{(1)}\otimes P_{N-n}^{(2)})\sigma_\mu(D_{-\frac{\nu}{2},k}) = A(N,n)|\mu|^k(P_n^{(1)}\otimes P_{N-n}^{(2)}) $$
with some polynomial $A(N,n)$ of order $\leq k$, depending on $\nu$. Hence
$$ P_n^{(1)}\circ\tr^{(2)}(\sigma_\mu(D_{-\frac{\nu}{2},k}f)) = |\mu|^k\sum_{N\geq n} A(N,n)\cdot\tr^{(2)}((P_n^{(1)}\otimes P_{N-n}^{(2)})T) $$
with $T=\sigma_\mu(f)$. Then
\begin{align*}
 & \|P_n^{(1)}\circ\tr^{(2)}(\sigma_\mu(D_{-\frac{\nu}{2},k}f))\|_{HS}\\
 ={}& \sup_{\|S\|_{HS}=1} |\langle P_n^{(1)}\circ\tr^{(2)}(\sigma_\mu(D_{-\frac{\nu}{2},k}f)),S\rangle|\\
 \leq{}& |\mu|^k\sup_{\|S\|_{HS}=1} \sum_{N\geq n} |A(N,n)|\cdot|\langle\tr^{(2)}((P_n^{(1)}\otimes P_{N-n}^{(2)})T),S\rangle|\\
 ={}& |\mu|^k\sup_{\|S\|_{HS}=1} \sum_{N\geq n} |A(N,n)|\cdot|\langle(P_n^{(1)}\otimes P_{N-n}^{(2)})T,S\otimes\1\rangle|\\
 ={}& |\mu|^k\sup_{\|S\|_{HS}=1} \sum_{N\geq n} |A(N,n)|\cdot|\langle(P_n^{(1)}\otimes P_{N-n}^{(2)})T,(P_n^{(1)}S\otimes P_{N-n}^{(2)})\rangle|\\
 \leq{}& |\mu|^k\sup_{\|S\|_{HS}=1} \sum_{N\geq n} |A(N,n)|\cdot\|(P_n^{(1)}\otimes P_{N-n}^{(2)})T\|_{HS}\cdot\|P_n^{(1)}S\otimes P_{N-n}^{(2)}\|_{HS}\\
 ={}& |\mu|^k\sup_{\|S\|_{HS}=1} \sum_{N\geq n} |A(N,n)|\cdot\|(P_n^{(1)}\otimes P_{N-n}^{(2)})T\|_{HS}\cdot\|P_n^{(1)}S\|_{HS}\cdot\|P_{N-n}^{(2)}\|_{HS}\\
 ={}& |\mu|^k\sum_{N\geq n} |A(N,n)|\cdot\|(P_n^{(1)}\otimes P_{N-n}^{(2)})T\|_{HS}\cdot\|P_{N-n}^{(2)}\|_{HS}\\
 \leq{}& |\mu|^m\left(\sum_{N\geq n}\lambda_N(\nu)\|(P_n^{(1)}\otimes P_{N-n}^{(2)})T\|_{HS}^2\right)^{\frac{1}{2}} \left(\sum_{N\geq n}\frac{|A(N,n)|^2\|P_{N-n}^{(2)}\|_{HS}^2}{\lambda_N(\nu)}\right)^{\frac{1}{2}}.
\end{align*}
Now $\|P_{N-n}^{(2)}\|_{HS}^2=\dim\mathcal{P}_{N-n}^{(2)}$ and
$$ \dim\mathcal{P}_{N-n}^{(2)} = {d+N-n-1\choose d-1} \sim (N-n+1)^{d-1} $$
with $d=\dim_\bc\fv_2=\frac{p-p_1}{2}$. Further note that by Remark~\ref{rem:CSNormOnFTSide}:
$$ \lambda_N(\nu) \sim (N+1)^{\rho-\nu}, \qquad \lambda^\flat_n(\nu+2k) \sim (n+1)^{\rho_1-\nu-2k}. $$
For $N\geq n$ we can moreover estimate $|A(N,n)|$ by $(N+1)^k$. Using
 Lemma~\ref{lem:StandardEstimate} 
we can then estimate the second sum:
\begin{align*}
 \sum_{N\geq n}\frac{|A(N,n)|^2\|P_{N-n}^{(2)}\|_{HS}^2}{\lambda_N(\nu)} &\leq C\sum_{N\geq n}\frac{(N+1)^{2k}(N-n+1)^{d-1}}{(N+1)^{\rho-\nu}}\\
 &= C \sum_\ell \frac{(\ell+1)^{d-1}}{(\ell+n+1)^{\rho-\nu-2k}}\\
 &\leq \frac{C'}{(n+1)^{\rho-\nu-2k-d}} \leq \frac{C''}{\lambda^\flat_n(\nu+2k)}.
\end{align*}
Inserting this above yields
\begin{align*}
 & \sum_n \lambda_n^\flat(\nu+2k)\|P_n^{(1)}\tr^{(2)}(\sigma_\mu(D_{-\frac{\nu}{2},k}f))\|_{HS}^2\\
 \leq{}& C''|\mu|^{2k}\sum_n\sum_{N\geq n} \lambda_N(\nu)\|(P_n^{(1)}\otimes P_{N-n}^{(2)})T\|_{HS}^2\\
 ={}& C''|\mu|^{2k}\sum_N\lambda_N(\nu)\|P_NT\|_{HS}^2
\end{align*}
which shows the claim.
\end{proof}

\section{Application to rank one semisimple Lie groups}

In the case where $N$ resp. $N_1$ is the nilradical of a parabolic subgroup in a rank one semisimple Lie group $G$ resp. $G_1$ the representations $X_\nu$ resp. $X_\nu^\flat$ extend to irreducible unitary representations of $G$ resp. $G_1$, the complementary series. In this section we explain how the invariant differential operators $\mathcal{D}_{\nu,k}$ construct discrete components $X_{\nu+2k}^\flat$ in the restriction of $X_\nu$ to $G_1$.

\subsection{Real semisimple Lie groups of rank one}\label{sec:RankOneGroups}

The connected real semisimple Lie groups of rank one are up to covering given by
$$ SO_0(1,n+1), \qquad SU(1,n+1), \qquad Sp(1,n+1), \qquad F_{4(-20)}. $$
Let $G$ be one of the these groups and let $P=MAN$ be a parabolic subgroup of $G$. Then the nilpotent radical $N$ is an H-type group and $NA$ its harmonic extension. In fact, the Lie algebra $\fn$ of $N$ is of the following form:
$$ \fn \simeq \begin{cases} \br^n & \mbox{for $G=SO_0(1,n+1)$,} \\ \bc^n\oplus\im\bc & \mbox{for $G=SU(1,n+1)$,} \\ \bh^n\oplus\im\bh & \mbox{for $G=Sp(1,n+1)$,} \\ \bo\oplus\im\bo & \mbox{for $G=F_{4(-20)}$.} \end{cases} $$

We realize $G$ explicitly as the identity component of the group of $(n+2)\times(n+2)$ matrices over $\boldf\in\{\br,\bc,\bh,\bo\}$ which preserve the quadratic form
$$ \boldf^{n+2}\to\br, \quad x\mapsto |x_1|^2-|x_2|^2-\cdots-|x_{n+2}|^2. $$
Here $n\geq1$ for $\boldf\in\{\br,\bc,\bh\}$ and $n=1$ for $\boldf=\bo$. Let $K$ be the maximal compact subgroup of $G$ which stabilizes the line $\boldf e_1$. Choosing
$$ H := \left(\begin{array}{ccc}0&1&\\1&0&\\&&\1_n\end{array}\right) $$
we can define a parabolic subgroup $P=MAN\subseteq G$ by putting
$$ M:=Z_K(H), \quad A:=\exp(\br H) \quad \text{and} \quad N:=\exp(\fn), $$
where $\fn$ is the sum of eigenspaces of $\ad(H)$ in $\fg$ corresponding to negative eigenvalues, namely $-1$ and possibly $-2$. Then $\fn$ identifies with $\boldf^n\oplus\im\boldf$ via
$$ \boldf^n\oplus\im\boldf\to\fn, \quad (X,Z)\mapsto\left(\begin{array}{ccc}Z/2&Z/2&X\\-Z/2&-Z/2&-X\\X^*&X^*&\0_n\end{array}\right). $$

We embed subgroups $G_1$ of the form
\begin{equation}
 SO_0(1,m+1), \qquad SU(1,m+1), \qquad Sp(1,m+1), \qquad Spin(8,1),\label{eq:RankOneSubgroups}
\end{equation}
respectively, into $G$ as blocks of size $(m+2)\times(m+2)$ in the upper left corner, i.e.
$$ G_1\hookrightarrow G, \quad g\mapsto\left(\begin{array}{cc}g&\\&\1_{n-m}\end{array}\right). $$
Then $P_1:=P\cap G_1$ is a parabolic subgroup for $G_1$ of the same type. Write $P_1=M_1AN_1$ then $N_1$ is an H-type subgroup of $N$ and $N_1A$ its harmonic extension. In the respective cases the subalgebras $\fn_1$ are of the form
$$ \fn_1 \simeq \begin{cases} \br^m & \mbox{for $G_1=SO_0(1,m+1)$,} \\ \bc^m\oplus\im\bc & \mbox{for $G_1=SU(1,m+1)$,} \\ \bh^m\oplus\im\bh & \mbox{for $G_1=Sp(1,m+1)$,} \\ \im\bo & \mbox{for $G_1=Spin(8,1)$.} \end{cases} $$

\begin{remark}
Strictly speaking the subalgebra
$\fn_1=\im\bo\subseteq\bo\oplus\im\bo=\fn$ is not of H-type according
to our definition since $\fv_1=0$ in this case and hence
$[\fv_1,\fv_1]\neq\fz$. However, the case $\fv_1=0$ can with some
notational modifications also treated by our methods.
 But the only possible bounded differential operator constructed in this way 
is the restriction operator.
This corresponds to the branching law for complementary series of
$F_{4(-20)}$ restricted to $Spin(8,1)$ which was already investigated 
by the third author, see \cite[Theorem 4.4]{Zha11} for the corresponding statement. Therefore we will not treat the case $\fv_1=0$ in this paper and always assume that $[\fv_1,\fv_1]=\fz$.
\end{remark}

\subsection{Spherical principal series representations of rank one groups}

Denote by $\overline{P}=MA\overline{N}$ the parabolic opposite to $P$. The Lie algebra $\fa=\br H$ of $A$ is one-dimensional and $\ad(H)$ has eigenvalues $+1$ and possibly $+2$ on $\overline{\fn}$, the Lie algebra of $\overline{N}$. We identify $\fa_\bc^*\simeq\bc$ via $\nu\mapsto\nu(H)$. For $\nu\in\fa_\bc^*$ let $\tau_{\nu}$ be the induced representation $\Ind_{\overline{P}}^G(\1\otimes e^\nu\otimes\1)$ of $G$, acting by left-translation on
$$ I(\nu)=\{f\in C^\infty(G):f(gme^{tH}\overline{n})=e^{-\nu t}f(g)\,\forall\,g\in G,me^{tH}\overline{n}\in MA\overline{N}\}. $$
(Our representation $\tau_\nu$ is
$Ind_{MA\overline{N}}^G(e^{\nu-\rho})$ in the standard
notation. However, our parameter $\nu$ has some advantage
over $\nu-\rho$ when dealing with branching since $G$ and $G_1$
have different $\rho$'s: our parameter $\nu$ is ``stable'' under branching, see Theorem~\ref{thm:BranchingRankOne} below.) Restricting $f\in I(\nu)$ to $N$ gives an $NA$-equivariant map to $C^\infty(N)$ endowed with the representation $\pi_\nu$ of $NA$, i.e.
$$ (\tau_\nu(g)f)|_N = \pi_\nu(g)(f|_N) \qquad \forall\,f\in I(\nu), g\in NA. $$

The unitarizable representations for real $\nu$ are usually called complementary series and have first been found by Kostant~\cite{Kos69}. Since the form \eqref{eq:InvForm} is also $G$-invariant we can immediately read off from Theorem~\ref{thm:ComplSerForHType} that $\tau_\nu$ is unitarizable for $\nu$ in the following interval:
\begin{align*}
 G &= SO_0(1,n+1), && 0 <\nu < n,\\
 G &= SU(1,n+1), && 0 <\nu < 2n+2,\\
 G &= Sp(1,n+1), && 2 <\nu < 4n+4,\\
 G &= F_{4(-20)}, && 6 <\nu < 16.
\end{align*}
In this case the unitary representation $X_\nu$ of $NA$ extends to an irreducible unitary representation of $G$.

Similarly we have smooth induced representations $(\tau_\nu^\flat,I^\flat(\nu))$ and complementary series representations $X_\nu^\flat$ of the subgroups $G_1\subseteq G$ in \eqref{eq:RankOneSubgroups}.

\subsection{Invariant differential operators and discrete components}

Consider the restriction of the complementary series representations $X_\nu$ to the subgroup $G_1\subseteq G$.

\begin{theorem}\label{thm:BranchingRankOne}
For $G=\SO_0(1,n+1)$, $SU(1,n+1)$ or $Sp(1,n+1)$ and $G_1=SO_0(1,m+1)$, $SU(1,m+1)$ or $Sp(1,m+1)$, respectively, the irreducible unitary representation $X_\nu$ of $G$ contains the irreducible unitary representation $X_{\nu+2k}^\flat$ of $G_1$ as a direct summand if $\nu>0$, $k\in\bn_0$ with $\rho-\frac{p}{2}-1<\nu<\rho$ and $\rho_1-\frac{p_1}{2}-1<\nu+2k<\rho_1$. The $G_1$-equivariant unitary embedding $X_{\nu+2k}^\flat\hookrightarrow X_\nu$ is realized by the adjoint of the operator $\mathcal{D}_{\nu,k}:X_\nu\to X_{\nu+2k}^\flat$.
\end{theorem}

\begin{proof}
By Theorem~\ref{thm:Boundedness} the operator $\mathcal{D}_{\nu,k}$ extends to a continuous linear operator $X_\nu\to X_{\nu+2k}^\flat$. It suffices to show that this operator is intertwining for the action of $G_1$. In fact, this implies that the adjoint $\mathcal{D}_{\nu,k}^*:X_{\nu+2k}^\flat\to X_\nu$ is intertwining, and since $X_{\nu+2k}^\flat$ is an irreducible unitary representation of of $G_1$ the operator $\mathcal{D}_{\nu,k}^*$ has to be a $G_1$-equivariant isometric embedding by Schur's Lemma.\\
To show that $\mathcal{D}_{\nu,k}:X_\nu\to X_{\nu+2k}^\flat$ is $G_1$-intertwining we use the identity \eqref{eq:DiffOpAsConjugationOfMult}. The Knapp--Stein intertwining operators $\Lambda_\nu$ and $\Lambda_{2\rho-\nu-2k}$ are intertwining for $G$ and hence for $G_1$, so it remains to show that the restriction operator $R$ and the multiplication operator $M_k$ are $G_1$-intertwining. For this consider the two operators
\begin{align*}
 \tilde{R}: {}&I(\beta)\to I^\flat(\beta), & \tilde{R}f(g_1) &= f(g_1),\\
 \tilde{M}_k: {}&I(\beta)\to I(\beta-2k), & \tilde{M}_kf(g) &= q(g)^k f(g),
\end{align*}
where the function $q\in C^\infty(G)$ is given by
$$ q(g)=\frac{1}{4}\left(|g_{m+3,1}+g_{m+3,2}|^2+\cdots+|g_{n+2,1}+g_{n+2,2}|^2\right), \qquad g\in G. $$
It is immediate from the definition of $I(\beta)$ that $\tilde{R}$ is defined and $G_1$-intertwining, i.e.
$$ \tilde{R}\circ\tau_\beta(g_1) = \tau_\beta^\flat(g_1)\circ\tilde{R} \qquad \forall\,g_1\in G_1. $$
Further, a short calculation shows that the function $q$ is left-invariant under $G_1$ and transforms under the right-action of $\overline{P}$ by $q(gme^{tH}\overline{n})=e^{2t}q(g)$. Therefore the operator $\tilde{M}_k$ is also defined and $G_1$-intertwining, i.e.
$$ \tilde{M}_k\circ\tau_\beta(g_1) = \tau_{\beta-2k}(g_1)\circ\tilde{M}_k \qquad \forall\,g_1\in G_1. $$
It is easy to see that $\tilde{R}$ and $\tilde{M}_k$ correspond to $R$ and $M_k$ in the sense that the following diagrams commute:
$$
\begin{xy}
\xymatrix{
 I(\beta) \ar[d]^{(\blank)|_N} \ar[r]^{\tilde{R}} & I^\flat(\beta) \ar[d]^{(\blank)|_{N_1}} && I(\beta) \ar[d]^{(\blank)|_N} \ar[r]^{\tilde{M}_k} & I(\beta-2k) \ar[d]^{(\blank)|_N}\\
 C^\infty(N) \ar[r]^{R} & C^\infty(N_1) && C^\infty(N) \ar[r]^{M_k} & C^\infty(N).
}
\end{xy}
$$
This implies that $R$ and $M_k$ are $G_1$-intertwining and the proof is complete.
\end{proof}

\appendix

\section{Jacobi polynomials}\label{app:JacobiPolys}

The classical Jacobi polynomials $P_n^{(\alpha,\beta)}(z)$ can be defined by (see \cite[equation 10.8~(12)]{EMOT53})
$$ P_n^{(\alpha,\beta)}(z) := \frac{1}{2^n}\sum_{m=0}^n{n+\alpha\choose m}{n+\beta\choose n-m}(x-1)^{n-m}(x+1)^m. $$
They satisfy the following parity identity (see \cite[equation 10.8~(13)]{EMOT53}):
$$ P_n^{(\alpha,\beta)}(-z) = (-1)^nP_n^{(\beta,\alpha)}(z). $$
The following recurrence relation in $n$ holds (see \cite[equation 10.8~(11)]{EMOT53}):
\begin{multline}
 (2n+\alpha+\beta+1)\left[(2n+\alpha+\beta)(2n+\alpha+\beta+2)z+\alpha^2-\beta^2\right]P_n^{(\alpha,\beta)}(z)\\
 = 2(n+\alpha)(n+\beta)(2n+\alpha+\beta+2)P_{n-1}^{(\alpha,\beta)}(z)+2(n+1)(n+\alpha+\beta+1)(2n+\alpha+\beta)P_{n+1}^{(\alpha,\beta)}(z).\label{eq:JacobiRecRel}
\end{multline}
We further have the following functional relation (see \cite[equation 10.8~(33)]{EMOT53}):
\begin{multline}
 (2n+\alpha+\beta+2)(1+z)P_n^{(\alpha,\beta+1)}(z)\\
 = 2(n+\beta+1)P_n^{(\alpha,\beta)}(z)+2(n+1)P_{n+1}^{(\alpha,\beta)}(z).\label{eq:JacobiSpecialKind}
\end{multline}
Identities \eqref{eq:JacobiRecRel} and \eqref{eq:JacobiSpecialKind} together yield
\begin{multline}
 2(n+\beta+1)\left[(2n+\alpha-\beta)+(2n+\alpha+\beta+2)z\right]P_n^{(\alpha,\beta)}(z)\\
 =(n+\alpha)(2n+\alpha+\beta+2)(1+z)^2P_{n-1}^{(\alpha,\beta+2)}(z)+4(n+1)(\beta+1)P_{n+1}^{(\alpha,\beta)}(z).\label{eq:NewIdentityJacobiPolynomials}
\end{multline}
We further remark that for $\alpha=-\frac{1}{2}$ or $\beta=-\frac{1}{2}$ the Jacobi polynomials degenerate to Gegenbauer polynomials (see \cite[equation 10.9~(21)]{EMOT53}):
\begin{equation}
 C_{2n}^\lambda(z) = \frac{(\lambda)_n}{(\frac{1}{2})_n}P_n^{(\lambda-\frac{1}{2},-\frac{1}{2})}(2z^2-1) = (-1)^n\frac{(\lambda)_n}{(\frac{1}{2})_n}P_n^{(-\frac{1}{2},\lambda-\frac{1}{2})}(1-2z^2).\label{eq:JacobiSpecializationGegenbauer}
\end{equation}

\bibliographystyle{amsplain}

\providecommand{\bysame}{\leavevmode\hbox to3em{\hrulefill}\thinspace}
\providecommand{\href}[2]{#2}

\end{document}